\documentclass[12pt,twoside]{article}

\usepackage[dvips]{graphics}
\usepackage{amsmath,amsfonts,amssymb,amsthm}
\usepackage{epsfig}
\usepackage{xspace}
\usepackage{color}

\newcommand{\dt}{\Delta t}
\newcommand{\dx}{\Delta x}
\newtheorem{thrm}{Theorem}
\newtheorem{lemma}{Lemma}
\newtheorem{defn}{Definition}

\newcommand{\mat}{\textsc{Matlab}\xspace}

\def\N{\mathbb N}

\def\R{\mathbb R}

\begin{document}
\title{Numerical convergence of a one step approximation of an
intrgro-differential equation}
\author{Samir Kumar Bhowmik~\footnote{The Author would like to thank
Professor Dugald B. Duncan for his valuable advise and kind help.
}
\\
 KdV Institute for Mathematics, University of Amsterdam\\
Amsterdam, Netherlands, S.K.Bhowmik@uva.nl\\ and \\
Department of Mathematics, University of Dhaka\\ Dhaka 1000, Bangladesh,
bhowmiksk@gmail.com.}
%
%
%
%
%
\maketitle

\begin{abstract}
We consider a linear partial integro-differential equation that arises in the modeling of various physical and biological processes. We study the problem in a spatial periodic domain. We analyze numerical stability and numerical
convergence of a one step approximation of the problem with smooth and non-smooth initial functions.
\end{abstract}
\section{Introduction}

Study in convolution model of phase transitions (initial value problem) is of ongoing
interest. Many scientific problems  have been modeled by reaction diffusion equation and advection reaction
 diffusion equations. A lot of models consider nonlocal type diffusion
 operators~\cite{A.Chmaj, C.Fife, SKB02, SamirKumarBhowmik02, Dug, Troy}. 
In this article we analyze stability, accuracy and rate of convergence of a simple approximation of such a linear partial integro-differential equation (IDE)
\begin{equation}\label{f:find-linearide01}
   u_{t}(x,t)=\varepsilon \int_{\Omega} J(x-y)\left( u(y,t)-u(x,t)\right) dy = \mathbb{L}u,
\end{equation}
which is the linear part of the IDE~\cite{A.Chmaj, C.Fife, PCFphase02, F.Chen, Dug1, Dug,
phase03, CFphase}
\begin{equation}\label{int01:f}
u_{t}=\varepsilon \left( \int_{\Omega}\ J(x-y) u(y,t)dy-u(x,t)
 \int_{\Omega}\ J(x-y)dy\right)+f(u).
\end{equation}
Here the initial condition
$u(x,0)=u_{0}(x), x\in \Omega$
where $\Omega \subseteq \mathbb{R}$, 
 $f(u)$ is a bistable nonlinearity for the associated ordinary differential equation
$u_{t}=f(u) $
and $J(x-y)$ is a kernel that measures interaction between particles at
positions $x$ and at $y$. Here it is assumed that the effect of close neighbours
$x$ and $y$ is greater than that from more distant ones;
the spatial variation is incorporated in $J.$
We assume that $J$ is a non-negative function satisfies smoothness, symmetry and decay conditions.

We present the stability and convergence of the rescaled integro-differential equation model in a periodic domain. Here
we consider spacial $[0, 1]$ periodic domain.
For simplicity of notations, from here $J$, $u$ mean functions in the periodic domain  and
$J^{\infty}$, $u^{\infty}$ mean functions  defined in the infinite domain.
 If we choose spatially  one-periodic
initial data $u(x,0)$, then for all $x\in \mathbb{R}$ and $t\in\mathbb{R_+}$
\[
u(x,t)=u(x+1,t).
\]
Then a with kernel function
$ J^{\infty}(x)$,  
(\ref{f:find-linearide01}) can be written as
\begin{eqnarray}\label{f:periodic01}
u_{t}&=& \varepsilon \int_{\mathbb{R}}J^{\infty} (x-y) \left(
u(y,t) - u(x,t)\right)dy\nonumber\\
%
&=&\varepsilon\int_{0}^{1}J(x-z)\left(u(z,t)-u(x,t)\right)dz,
\end{eqnarray}
where
\begin{equation}\label{f:perdJ}
J(x)= \sum_{r=-\infty}^{\infty} J^{\infty} (x- r)
\end{equation}
and $x\in [0, 1].$
%
%
 Two sample kernels are shown in
Figure~{\ref{ker01a:fig}}.  Here we consider $\varepsilon = 1$ only since it does not affect our analysis.
\begin{figure}[t]
     \begin{center}
       {\bf (a) \hspace{5.5cm} (b)}\\
       \includegraphics[width=0.49\textwidth,height=6.5cm]{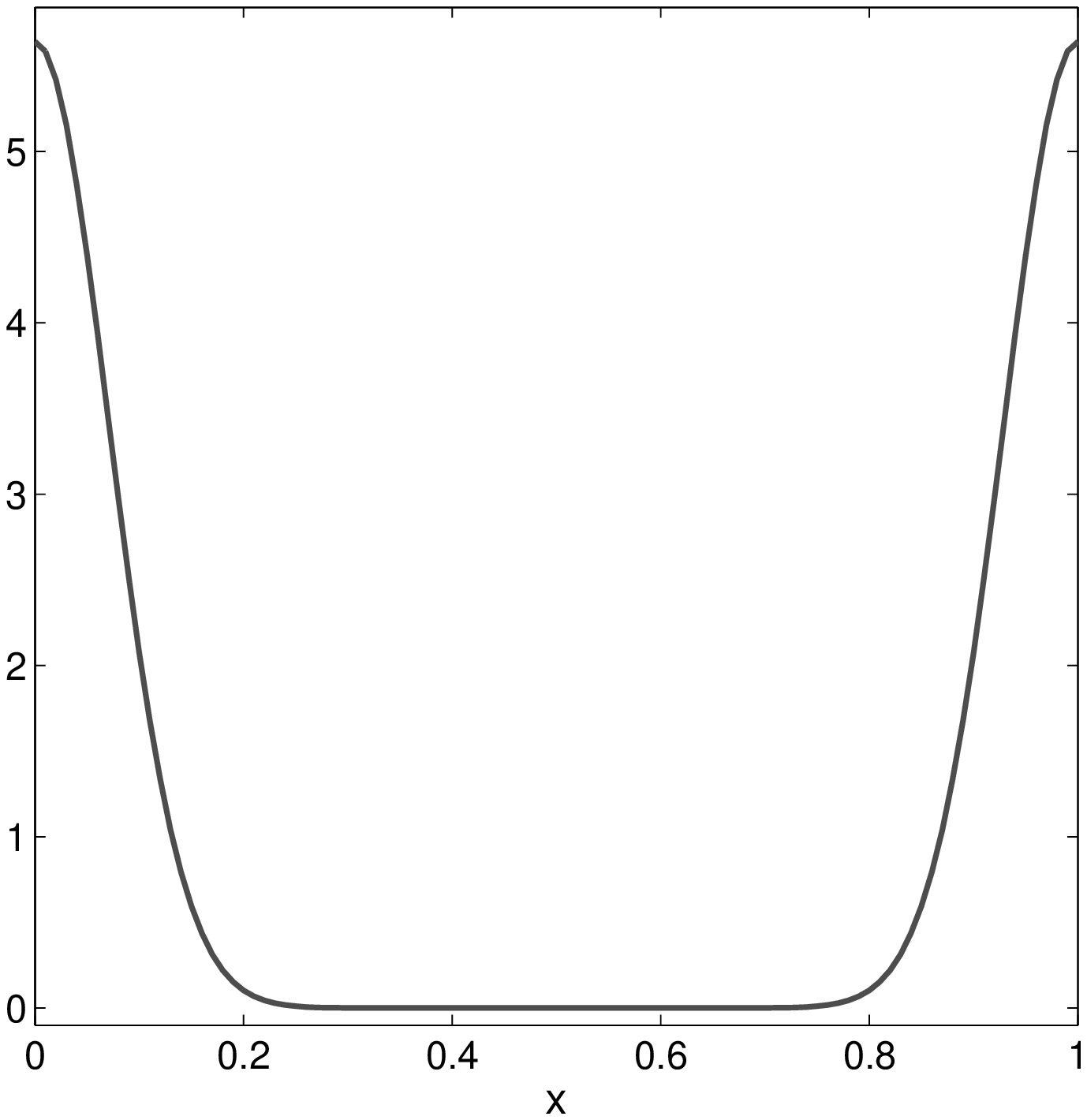}
       \includegraphics[width=0.49\textwidth,height=6.5cm]{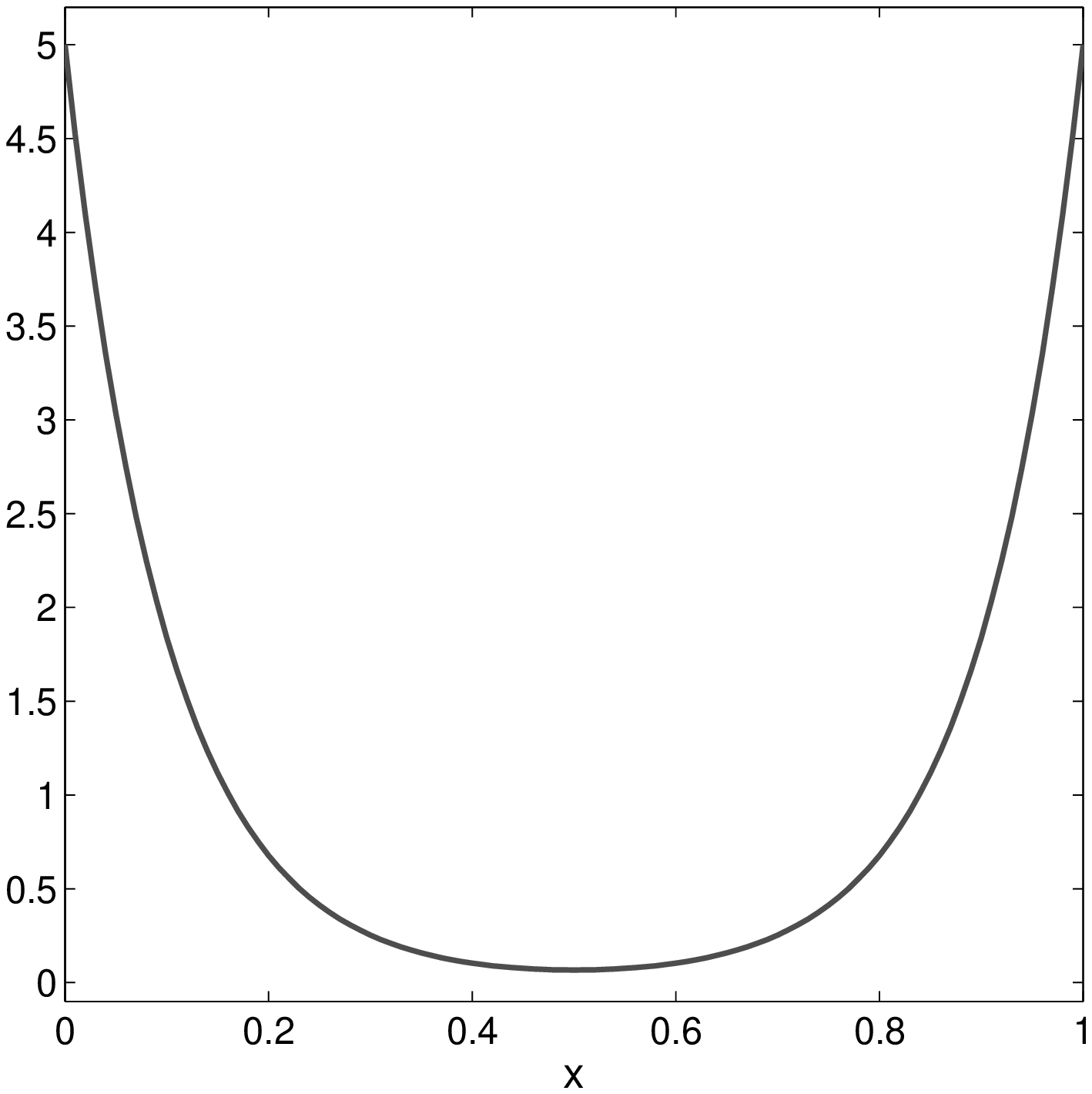}
      \end{center}
     \caption{Examples of kernel functions in $[0, 1]$ periodic domain defind by (\ref{f:perdJ}) where  (a) $J^{\infty}(x)=\sqrt{\frac{100}{\pi}} e^{-100 x^2}$, (b) $J^{\infty}(x)=\frac{c}{2} e^{-c|x|}$, $c=10$.
}
     \label{ker01a:fig}
    \end{figure}

In \cite{SamirKumarBhowmik02}, the author study stability and accuracy of an one step approximation the IDE considering infinite spatial domain. They  first write a discrete equivalent of the integral equation, then use the forward Euler method for time discretisation. They present some numerical results to demonstrate the rate of convergence for smooth and non-smooth initial functions.

In \cite{Dug}, Duncan et. al. consider (\ref{int01:f}) in a spatial periodic $[0, 1]$ and $[0, 1]^2$ domains. They approximate the problem using piecewise constant basis
function with collocation and mid point quadrature rule for space discretisation, then they use some standard ode solver for time integration. They present some numerical results to demonstrate their scheme.
In \cite{DJD001}, the author consider several local and nonlocal operators that contains (\ref{f:find-linearide01}). He approximate the models using finite difference schemes and discuss some stability issues.

Here we consider the problem  (\ref{f:periodic01}) in a spatial periodic domain, and study stability and accuracy of a simple scheme using the Fourier series and the discrete Fourier transform definitions.
We Organize our article in the following way. We present a simple
space time discretisation in Section~2 followed by stability analysis of the scheme in Section~3.
We analyze accuracy of the scheme in Section~4 considering smooth initial function,
whereas in Section~5, we discuss the same considering non-smooth initial function.
We study  accuracy
of the semi-discrete time dependent scheme in Section~6. We finish our study in Section~7 with some
numerical experiments and discussions.
%
\section{Numerical approximation}
We approximate (\ref{f:find-linearide01}) as
\begin{equation}\label{period02:f}
\frac{dU_{j}(t)}{dt}=h \sum_{k=0}^{N-1} J(x_j-x_k)(U_k-U_j)
\end{equation}
for each $j=0, 1, 2,\cdots,N-1$ where $U_{j}(t)\approx u(x_j,t)$ and $x_j=j h.$
Applying Euler's method we can again approximate the semi-discrete
problem (\ref{period02:f}) as
\[
U_{j}^{n+1}-U_{j}^{n}= \dt h \sum_{k=0}^{N-1} J(x_j-x_k)
\left( U_{k}^{n}-U_{j}^{n} \right)
\]
\begin{equation}\label{period03:f}
U_{j}^{n+1}=U_{j}^{n}\left(1-
 h \dt  \sum_{r=0}^{N-1} J(x_j-x_{r})\right)+
h \dt \sum_{r=0}^{N-1} J(x_j-x_{r})U_{r}^{n}.
\end{equation}

We use Fourier series and discrete
Fourier transforms throughout. Before the main discussion let us introduce some necessary definitions and theorems.
%
%
%
The Discrete Fourier Transform for a periodic function (DFT) can be
 defined as
\[
\tilde U_k=h\sum_{j=0}^{N-1} u(x_j)e^{{-i2\pi kx_j}}
\]
 where $u(x,t)$ is a periodic function with period $1$
and its inverse Fourier transform is defined as
\[
u_j= h\sum_{k=0}^{N-1} \tilde U_k e^{{i2 \pi kx_j}}
\quad \text{or} \quad
u_j= h\sum_{k=-\frac{N}{2}+1}^{\frac{N}{2}} \tilde U_k e^{{i2\pi kx_j}}.
\]
For any function periodic function $f(x)$ with period $1,$ its Fourier
series can be defined as
\[
f(x)=\sum_{n=-\infty}^{\infty} \hat f_n e^{{i2\pi nx}}
\quad \text{where}
%
\quad
\hat f_n= \int_{0}^{1} f(x)e^{{-i2\pi nx}}dx.
\]
For any $1$-periodic  complex-valued function $u(x)$ the relation
\[
\sum_{n=-\infty}^{\infty} |\hat u_n|^2=\int_0^{1} |u(x)|^2 dx
\]
where $\hat u_n$ are the Fourier coefficients of $u(x)$
 and
 for a real-valued function $u(x)$
\[
\frac{\hat u_0 ^2}{4} +\frac{1}{2}\sum_{n=1}^{\infty} \left( \hat u_{-n} ^2 + \hat u_n ^2\right)=
\int_0^{1} u^2(x) dx
\]
is called Parseval's relation for the Fourier series where
$a_n$, $b_n$ are Fourier coefficients.
 For the Discrete
Fourier Transform of $u(x)$, the relation
\[
\sum_{j=0}^{N-1}|u_j|^2=h \sum_{k=0}^{N-1}|\tilde U_k|^2 \quad or \quad
\mbox{ $h \sum_{k={-\frac{N}{2}+1}}^{\frac{N}{2}}|\tilde U_k|^2$}
\]
is called Parseval's relation where $\tilde U_k $ is the DFT of
 $u_j,$  both have the
same length $N.$
%

Now we will discuss the relation between the DFT and the Fourier coefficients.
From the DFT definition
$
\tilde U_k = h\sum_{j=0}^{N-1} u(jh)e^{{-i2\pi k x_j}},
$   
the
Fourier coefficients are
\[
\hat u_k= \int_{0}^{1} u(x)e^{{-i2\pi k x}}dx
\]
 \text{and the corresponding Fourier series is}
\[
u(x)=\sum_{k=-\infty}^{\infty} \hat u_k e^{{i2\pi k x}}.
\]
Thus
\begin{eqnarray}\label{poisson:f}
\tilde U_k &=&h\sum_{j=0}^{N-1} \sum_{r=-\infty}^{\infty}
\hat u_r e^{{i2\pi r x_j}}
 e^{{-i2\pi k x_j}} \nonumber\\
&=& h \sum_{r=-\infty}^{\infty} \hat u_r N \delta_N (r-k)
= \sum_{m=-\infty}^{\infty} \hat u_{k+mN}
\end{eqnarray}
where $r-k=mN$ and
\begin{equation}\label{def:deltafunc}
\delta_N (j)=\left\{
              \begin{array}{ll}
               1 & \mbox{$j=mN$ for some $m\in \mathbb{Z}$}\\
               0 & \mbox{otherwise.}
               \end{array}
             \right.
\end{equation}
The relation (\ref{poisson:f}) is known as the
discrete Poisson sum formula~\cite{Str}.
The Fourier coefficients of the kernel function in $[0, 1]$ are given by
%
%
%
\begin{equation}\label{Pnprela:f}
\hat J_j=\int_{0}^{1} J(x)e^{{-i2\pi j x}}dx= \int_{-\infty}^{\infty} J^{\infty} (y) e^{{-i2\pi j y}}dy
=\hat J^{\infty} \left ({2\pi j}\right )
\end{equation}
  is the relation between the Fourier coefficient of $J(x)$ defined in (\ref{f:perdJ})
  and the continuous Fourier transform of $J^{\infty}(x)$.

Multiplying both sides of (\ref{period03:f}) by $h e^{{-i2\pi k x_j}}$
and summing over $j$
\begin{eqnarray*}
h\sum_{j=0}^{\N-1} U_j^{n+1} e^{{-i2\pi k x_j}}&=&
h\sum_{j=0}^{N-1} U_j^{n} e^{{-i2\pi k x_j}}\left (
1-h\dt \sum_{r=0}^{N-1} J(x_j-x_{r})\right )\\
&&+ h \sum_{r=0}^{N-1} e^{{-i2\pi k x_j}} U_{r}^n\left (
h \dt \sum_{j=0}^{N-1} J(x_j-x_{r})e^{{-i2\pi k x_j}}
\right ).
\end{eqnarray*}
That is,
\begin{eqnarray*}
\tilde U_k^{n+1}&=&\tilde U_k^{n} \left (
1-h\dt \sum_{r=0}^{N-1} J(x_j-x_{r})+ h \dt \sum_{j=0}^{N-1} J(x_j-x_{r}) e^{{-i2\pi k x_j}}
\right )\\
&=& g(h,\dt,k)\tilde U_k^{n}
\end{eqnarray*}
and clearly
\begin{equation}\label{period04:f}
\tilde U_k^{n}=g^n (h,\dt,k) \tilde U_k ^0,
\end{equation}
where
\begin{eqnarray*}
g(h,\dt,k)&=&
1+h \dt \sum_{j=0}^{N-1} J(x_j-x_{r}) \left(
e^{{-i2\pi k x_j}}-1\right)
= 1+\dt \left( \tilde J_k-\tilde J_0 \right)
\end{eqnarray*}
since $J$ is $1$-periodic. 

\section{Stability analysis}
To show the stability of the scheme we need some reasonable restrictions on $J(x)$.
We will examine properties of the DFT  of $J(x)$ under some reasonable
hypothesis on the function $J(x)$.  We use the bounds obtained below to get
a bound on $g(h, \dt, k).$
\begin{lemma}\label{periodlemma01}
Assume that
\begin{description}
\item[A1.] $J(x) \ge 0.$
\item[A2.] $J(x)=J(-x).$
\item[A3.]  $\int_{\Omega} J(x) dx = 1.$
\item[A4.]  $\frac{d}{dx} J(x) < 0$ for  $x \in (0, \frac{1}{2})$
\item[A5.]  $\hat J_k \ge 0$ for  $-\frac{N}{2}+1 \le k \le \frac{N}{2}$
\end{description}
then  $0 \le \tilde J_0-\tilde J_k \le 2$, $\hat J_0-\hat J_k \le 2$ where $-\frac{N}{2}+1 \le k \le \frac{N}{2}$. Furthermore,
$\tilde J_k \ge 0.$
\end{lemma}
\begin{proof}
We have
\[
\tilde J_k= h\sum_{r=0}^{N-1} J(x_r) e^{-{ik2\pi x_r}}
\quad \text{and} \quad
\tilde J_0= h\sum_{r=0}^{N-1} J(x_r).
\]
So
\[
\tilde J_0-\tilde J_k=
h\sum_{r=0}^{N-1} J(x_r) \left (1-e^{-{ik2\pi x_r}}\right).
\]
For simplicity from here we take $N$ to be even.
Using the symmetry of $J(x)$ in $[0,  1]$
\[
\tilde J_0-\tilde J_k=
 2 h\sum_{r=0}^{\frac{N}{2}} J(x_r)
\left ( 1-\cos ({k2\pi x_r})\right )\ge 0
\]
as $h>0$ and $J\ge 0.$
Also $1-\cos ({k2\pi x_r})\le 2.$
Thus
\[
\tilde J_0-\tilde J_k=
 4 h\sum_{r=0}^{\frac{N}{2}} J(x_r)
\le 4\int_0^{\frac{1}{2}} J(x)dx=2\int_0^{1} J(x)dx = 2.
\]
So we conclude
$ 
0 \le \tilde J_0-\tilde J_k\le 2.
$   
 The result $\hat J_0-\hat J_k \le 2$  follows from the definition of $J$ and
 by using similar steps to those of \cite{SKB02}.
\end{proof}
\begin{lemma}\label{periodlemma02}
If $J(x)$ satisfies \textbf{A1}~-~\textbf{A3} then
$
1 \ge g(h, \dt, k)\ge 1-C \dt
$ 
for some $C>0$ and if $J(x)$ satisfies \textbf{A4}~-~\textbf{A5} from
Lemma~\ref{periodlemma01} as well then $C=2$ and
\[
|g(h,\dt,k)|\le 1 \quad \mbox{for all} \quad \mbox{$0<\dt\le \dt^*=\frac{2}{C}.$}
\]
\end{lemma}
\begin{proof}
Proof of this Lemma follows directly  from Lemma~\ref{periodlemma01}.
\end{proof}
The stability result follows from the following theorem.
\begin{thrm}\label{periodthrm03}
If $J(x)$ is a  periodic function in $[0, 1] $ and satisfies \textbf{A1}~-~\textbf{A5},
then there exists $\dt^*>0$ given by Lemma~\ref{periodlemma02}
such that
$
\|U^n\|_h\le \|U^0\|_h
$ 
for all $0<\dt\le \dt^*$ and $n\ge 0.$
\end{thrm}
\begin{proof}
We have
\begin{eqnarray*}
\|U^n \|_h^{2}&=& h\sum_{j=0}^{n-1}|U_j^n|^2= h^2\sum_{j=0}^{N-1}
 \left |\tilde U_j^n
 \right |^2
=
 h^2\sum_{j=0}^{N-1}
 \left | g^n (h,\dt,k) \tilde U_j^0
\right |^2\\
&\le&
 h^2\sum_{j=0}^{N-1}
 |\tilde U_j^0|^2
\le
 h\sum_{j=0}^{N-1}
| U_j^0|^2=\|U_j^0\|_h^2,
\end{eqnarray*}
which gives the result.
\end{proof}
%
%
%
%
\section{Convergence analysis of the fully discrete approximation}\label{defineqqq}
Let the Fourier series of $u(x,t)$ be
\begin{equation}\label{period05:f}
u(x,t)=\sum_{j=-\infty}^{\infty} \hat u_j (t) e^{{ij2\pi x}}
\end{equation}
where
\[
\hat u_j (t) = \int_0^{1} u(x,t)  e^{-{ij2\pi x}}dx
\]
and let the Fourier series expansion of $J(x)$ be
\begin{equation}\label{period06:f}
J(x)=\sum_{j=-\infty}^{\infty} \hat J_j e^{{ij2\pi x}}
\end{equation}
where
\[
\hat J_j =\int_0^{1} J(x)  e^{-{ij2\pi x}}dx.
\]
Substituting (\ref{period05:f})-(\ref{period06:f})
in (\ref{f:find-linearide01})
\begin{eqnarray*}
\sum_{j=-\infty}^{\infty} \frac{d}{dt}\hat u_j (t) e^{{ij2\pi x}}&=&
 \left ( \sum_{j=-\infty}^{\infty} \hat u_j (t) \hat J_j e^{{ij2\pi x}}
-\sum_{j=-\infty}^{\infty} \hat u_j (t) e^{{ij2\pi x}} \right ) \\
\frac{d}{dt} \hat u_j(t) &=&  (\hat J_j- \hat J_0) \hat u_j (t)
                         = \hat q_j \hat u_j (t)
\end{eqnarray*}
with $\hat q_j= (\hat J_j- \hat J_0)$ where $\hat J_j$ is the $j$th Fourier coefficient
of the kernel function $J(x).$
Solving the above equation we have
\[
\hat u_j(t)=e^{\hat q_j t}\hat u_j (0),
\]
where
\[
\hat u_j (0)= \int_0^{1} u(x,0) e^{{-ij2\pi x}}dx.
\]
Thus the exact solution of the IDE (\ref{f:find-linearide01}) can be written as
\begin{equation}\label{period07:f}
u(x,t)=\sum_{j=-\infty}^{\infty} \hat u_j (0)e^{\hat q_j t} e^{{ij2\pi x}}.
\end{equation}
%
Now taking the inverse of the discrete Fourier transform in (\ref{period04:f})
we get the approximate solution of (\ref{f:find-linearide01}) as
\begin{equation}\label{period08:f}
U_j^{n}=\sum_{k=-\frac{N}{2}+1}^{\frac{N}{2}}
g^n (h,\dt,k)\tilde U_k^0 e^{{ik2\pi x_j}}.
\end{equation}
From (\ref{period07:f})
and (\ref{period08:f})
\begin{eqnarray}\label{period09:f}
u(x_j,t_m)-U_j^{m} &=& \sum_{k=-\infty}^{\infty} \hat u_k(0) e^{\hat q_k t_m}
 e^{{i2\pi k x_j}} 
-\sum_{k=-\frac{N}{2}+1}^{\frac{N}{2}}
g^m (h,\dt,k)\tilde U_k^0 e^{{ik2\pi x_j}}\nonumber\\
 &=&
 \sum_{k=-\frac{N}{2}+1}^{\frac{N}{2}}
 \left ( \hat u_k(0) e^{\hat q_k t_m}-g^m (h,\dt,k)\tilde U_k^0\right )
 e^{{i2\pi k x_j}} \nonumber \\
%
&&+\sum_{|k|>\frac{N}{2}} \hat u_k(0) e^{\hat q_k t_m} e^{{ik2\pi x_j}}
\end{eqnarray}
Now taking the inner product on (\ref{period09:f}) and
applying Parseval's relation
\begin{eqnarray}\label{period10:f}
\| u(x_j,t_m)-U_j^{m}\|_h^2 &\le&
 \sum_{k=-\frac{N}{2}+1}^{\frac{N}{2}}
 \left | \hat u_k(0) e^{\hat q_k t_m}-g^m (h,\dt,k)\tilde U_k^0\right |^2
 \nonumber \\
&&+\sum_{|k|>\frac{N}{2}}\left | \hat u_k(0) e^{\hat q_k t_m} \right |^2.
\end{eqnarray}
The Poisson summation formula gives
\[
\tilde U_k^{m}=\sum_{r=-\infty}^{\infty}\hat u_{k+r N}^{m}
=\sum_{r=-\infty}^{\infty}\hat u_{k+r N}^{\infty,m}
\]
where $\hat u_{k+r N}^{m}$  is the coefficient of the Fourier series at $t_m$ for the
periodic function $u$ whereas $\hat u_{k+r N}^{\infty,m}$ represents the CFT for
the nonperiodic (infinite-dimensional) case.
A calculation similar to that leading to (\ref{Pnprela:f})
gives
\begin{equation}\label{finiteinfinite:f}
\hat u_k (0)=\int_0^{1} u (x,0) e^{{-ik2\pi x}}dx
=\hat u_k^{0}
=\hat u^{\infty}_0\left ( {2\pi k} \right )
=\hat u^{\infty,0}_k.
\end{equation}
So the first part of the right-hand side of (\ref{period10:f}) can be written as
\begin{eqnarray}\label{period10a:f}
&&\sum_{k=-\frac{N}{2}+1}^{\frac{N}{2}}
\left | \hat u_k (0) e^{\hat q_j t_m} - g^m \tilde U_{k}^{0} \right |^2
 \le
\sum_{k=-\frac{N}{2}+1}^{\frac{N}{2}}
\left | \left ( e^{\hat q_j t_m} - g^m \right )
\hat u_{k}^{\infty,0} \right |^2
\nonumber\\
&&\qquad\qquad +\qquad
\sum_{k=-\frac{N}{2}+1}^{\frac{N}{2}} \left |
g^m (h,\dt,k)\sum_{r\ne 0}\hat u^{\infty,0}
\left ( {2\pi(k+rN)}\right )
\right |^2.
\end{eqnarray}
Before giving the main convergence results let us introduce a lemma with
reasonable restrictions on $J(x).$
\begin{lemma}\label{periodlemma03}
If $J(x)$ is a nonnegative even periodic function
in $[0, 1]$ and  monotone-decreasing in $[0, \frac{1}{2}]$ then
$
\hat q_j \le 0
$  and $|\hat J_0- \hat J_j|\le \frac{1}{2}$
where $\hat q_j=(\hat J_j-\hat J_0)$ with
$\hat J_j= \int_{0}^{1} J(x) e^{{-ij2\pi x}}dx.$
\end{lemma}
\begin{proof}
The proof of Lemma~\ref{periodlemma03} follows directly from the definition of $\hat J_j.$
\end{proof}
 Now
\[|e^{\hat q_j \dt}|\le 1 \quad \mbox{ and} \quad |g(h, \dt,k)|\le 1.
\] So
\[|e^{\hat q_k t_m}-g^m (h,\dt,k)|=|\left (e^{\hat q_k \dt}\right )^m-g^m(h, \dt, k)|
\le m|e^{\hat q_k\dt}-g(h,\dt,k)|.
\]
Also,
\begin{eqnarray}\label{period12:f}
&&\qquad e^{\hat q_k\dt}-g(h,\dt,k)\nonumber\\
&=& e^{\dt \left ( \hat J_k - \hat J_0 \right )}-
\left(1+\dt \left( \tilde J_k -\tilde J_0 \right )   \right)
\nonumber \\
&=& \dt \left ((\hat J_k-\hat J_0)-\left( \tilde J_k -\tilde J_0 \right )\right )
 +\sum_{j=2}^{\infty} \frac{\dt ^j}{j!}
\left (\hat J_k-\hat J_0\right )^j,
\end{eqnarray}
and thus there exists $C_1(h)$ and $C_2$ (using (\ref{poisson:f}) and (\ref{Pnprela:f})) such that
\[
\left | e^{\hat q_j \dt}-g(h, \dt, k) \right |\le \dt \left ( C_1 (h)+C_2 \dt \right ).
\]
Here $C_1(h)\rightarrow 0$ as $h\rightarrow 0$,  and $C_2$ is bounded, see \cite{SKB02} for exact detail.
Thus we get the following bound.
\begin{thrm}\label{pdefn04}{Accuracy:}
If $J(x)\in H^r(1)$, $r > \frac{1}{2}$ satisfies \textbf{A1}~-~\textbf{A5} then
there exist $C_1 (h)$ and $C_2$
such that
\[
\left | e^{\hat q_j \dt}-g(h, \dt, k) \right |\le \dt \left ( C_1 (h)+C_2 \dt \right ).
\]
\end{thrm}
Applying Theorem~\ref{pdefn04} on the first part of the right-hand side of (\ref{period10a:f}) we get
\begin{equation}\label{period13:f}
\sum_{k=-\frac{N}{2}+1}^{\frac{N}{2}}
\left |(e^{\hat q_j t_m}-g^m) \hat u_0^{\infty}\right |^2
\le \dt^2\left ( C_1 (h)+C_2 \dt \right )^2 \|u_0\|^2.
\end{equation}
\begin{defn}\label{pdefn03a}~\cite[page 223]{Atkinson}
For  integer $ k \ge 0,$ $H^{k} (2\pi)$ is defined to be the closure of
$C_{p}^{k} (2\pi)$ under the inner product norm
\[
\|\varphi\|_{H^{k}}=\left [ \sum_{j=0}^{k} \|\varphi ^{(j)}\|_{L^2}^2  \right ]^{\frac{1}{2}}.
\]
For arbitrary real $s \ge0$, $H^s(2\pi)$ can also be obtained following~\cite[pages, 219-223]{Atkinson}.
\end{defn}
\begin{thrm}\label{soveedef:f}~\cite[page 223]{Atkinson}
For $s\in \mathbb{R},$ $H^s (2\pi)$ is the set of all series
\[
\varphi (x)=\sum_{m=-\infty}^{\infty}a_m \psi _m (x)
\]
\text{for which} 
\[
\|\varphi\|^2_{*,s}=|a_0|^2+\sum_{|m|>0}|m|^{2s}|a_m|^2<\infty
\]
     where
\[
\psi_m (x) = \frac{1}{\sqrt{2 \pi}} e^{imx}, \qquad m=0, \pm1, \pm2, \cdots.
\]
 Moreover, the norm $||\varphi||_{*,s}$ is equivalent to the standard Sobolev norm
$||\varphi||_{H^s}$ for $\varphi \in H^s (2\pi).$
\end{thrm}
For exact details of the Theorem~\ref{soveedef:f} please see~\cite[page 223]{Atkinson}.
In this section, to use norm definitions in a periodic domain $[0, 1]$, we use
$\|\varphi\|_{H^s} =\|\varphi\|_{{H^s}(1)}$.

Now
\begin{eqnarray*}
&&\left |\sum_{s\ne 0}\hat u^{\infty,0}
\left ( {2\pi(k+sN)}\right )
\right |
\le\sum_{s\ne 0}\left |\hat u^{\infty,0}
\left ( {2\pi(k+sN)}\right)\right |\\
&\le& \sqrt{ \sum_{s\ne 0}\left |\hat u^{\infty,0}
\left ( {2\pi(k+sN)} \right )
\left ( {2\pi(k+sN)} \right )^{\sigma}\right|^2}
\sqrt{\sum_{s\ne 0}
\left |\left ( {2\pi(k+sN)} \right )\right |^{-2\sigma}}
\end{eqnarray*}
using the Cauchy-Schwartz inequality.
Thus,
\begin{eqnarray*}
&&\sum_{k=-\frac{N}{2}+1}^{\frac{N}{2}} \left |
g^m (h,\dt,k)\sum_{s\ne 0}\hat u^{\infty,0}
\left ( {2\pi(k+sN)}\right )
\right |^2
%
\le \sum_{k=-\frac{N}{2}+1}^{\frac{N}{2}} \left |
\sum_{s\ne 0}\hat u^{\infty,0}
\left ( {2\pi(k+sN)}\right )
\right |^2
%
\\
&\le&
\sum_{k=-\frac{N}{2}+1}^{\frac{N}{2}}
\left (\sqrt{ \sum_{s\ne 0}\left |\hat u^{0}
\left ( {2\pi(k+sN)} \right )
\left ( {2\pi(k+sN)} \right )^{\sigma}\right|^2}\right)
\left (
\sqrt{\sum_{s\ne 0}
\left |\left ( {2\pi(k+sN)} \right )\right |^{-2\sigma}}\right )
\end{eqnarray*}
 gives
\begin{eqnarray}\label{period11:f}
&&\sum_{k=-\frac{N}{2}+1}^{\frac{N}{2}} \left |
g^m (h,\dt,k)\sum_{s\ne 0}\hat u^{\infty,0}
\left ( {2\pi(k+sN)}\right )
\right |^2
\nonumber\\
%
&\le& h^{2\sigma} C(\sigma)\sum_{|k|>\frac{N}{2}}
\left | {2\pi k}\right |^{2\sigma}
\left |
\hat u^{0}\left ({2\pi k}\right )
\right |^2
\nonumber\\
&\le& C(\sigma) h^{2\sigma}\| u_0\|_{*,\sigma}^2
\equiv C(\sigma) h^{2\sigma}\| u_0\|_{H^{\sigma}(1) }^2
\end{eqnarray}
for some $\sigma>\frac{1}{2}.$
%
%
%
Thus applying above bounds, Lemma~\ref{periodlemma03}
 and Theorem~\ref{soveedef:f} in (\ref{period10:f}) one gets
\begin{eqnarray}\label{period13:f}
\| u(x_j,t_m)-U_j^{m}\|_h^2 &\le&
\dt^2 \left (C_1 (h)+C_2 \dt \right )^2\|u_0\|^2+
C_3 h ^{2\sigma}
\|u_0\|_{H^\sigma (1)}^2
+\sum_{|k|>\frac{N}{2}}\left | \hat u_k(0) \right |^2 \nonumber\\
&\le& \dt^2 \left (C_1 (h)+C_2 \dt \right )^2\|u_0 \|^2+
C_3 h ^{2\sigma}
\|u_0 \|_{H^\sigma (1)}^2
\nonumber \\
&&+\left (\frac{2}{N}\right )^{2\sigma}
\sum_{|k|>\frac{N}{2}} |k|^{2\sigma}
\left | \hat u_0 \left({2k\pi}\right) \right |^2
\nonumber \\
&\le& \dt^2 \left (C_1 (h)+C_2 \dt \right )^2\|u_0 \|^2+C_3 h ^{2\sigma}
\|u_0 \|_{H^\sigma (1)}^2, \nonumber\\
\end{eqnarray}
for some $\sigma>\frac{1}{2}$. Thus we conclude this discussion with the following result.
\begin{thrm}\label{f:thrm001222}
If the approximation (\ref{period03:f}) of the  initial
value problem (\ref{f:find-linearide01}) is stable, $J, \hat J_k,$ $-\frac{N}{2}+1 \le k\le \frac{N}{2}$ satisfy
 \textbf{A1}~-~\textbf{A5} and
 $u_{0}\in H^{\sigma}({1})$ with $\sigma>\frac{1}{2},$ then there exist
constants $C_{1}(h),$ $C_2,$ $C_{3}(\sigma)$ such that
\[
\|u(x,t_m)-U_{j}^{m} (x)\|\le \dt \left(C_1 (h)
+C_2  \dt \right) \|u_0\| +
C_3(\sigma) h^{\sigma}\|u_0\|_{H^{\sigma}(1)}.
\]
\end{thrm}
\section{Convergence analysis for  non-smooth initial data}
Now let us consider an initial function which is not smooth enough so that
$\|u_0\|_{H^{\nu} (1)}$ is bounded  when $\nu >\frac{1}{2},$ but there exists $\nu_1<\nu$ such that  $\|u_0\|_{H^{\nu_1} (1)}< \infty$ and there exists $\alpha $ such that
\begin{equation}\label{eqq110:f}
\sum_{k = - \frac{N}{2}+1}^{\frac{N}{2}}
\left |
\hat u^{\infty,0} \left ({2\pi(k + sN )} \right )
\right |^2
\le \dx^{2 \alpha} C(u_0).
\end{equation}
For example of such functions please see \cite{SKB02, Str} and references there in.
We start varying (\ref{period10:f}) as
\begin{eqnarray*}
\|u(x_j, t_m) - U_j^m\|_h^2 &\le&
\sum_{|k|\le M}
\left|
\hat u_k (0) e^{\hat q_k t_m} - g^m \tilde U_k^0
\right|^2\\
&&+
\sum_{M\le |k|\le \frac{N}{2}}
\left|
\hat u_k (0) e^{\hat q_k t_m} - g^m \tilde U_k^0
\right|^2
+\sum_{|k|>\frac{N}{2}}
\left|
\hat u_k (0) e^{\hat q_k t_m}
\right|^2
\end{eqnarray*}
for some $0<M<\frac{N}{2}.$
Now with the same algebraic operations as we performed for the Theorem~\ref{f:thrm001222}
\begin{eqnarray*}
\|u(x_j, t_m) - U_j^m\|_h^2
 &\le&
\sum_{|k|\le M}
\left|
\hat u_k (0) e^{\hat q_k t_m} - g^m \hat u_k^{\infty, 0}
\right|^2
+
\sum_{M\le |k|\le \frac{N}{2}}
\left|
\hat u_k (0) e^{\hat q_k t_m} - g^m \hat u_k^{\infty,0}
\right|^2
\\
&&
+\sum_{|k|>\frac{N}{2}}
\left|
\hat u_k (0) e^{\hat q_k t_m}
\right|^2
+
\sum_{|k|\le M}
\left|
g^m \sum_{s \ne 0} \hat u^{\infty, 0}\left ( {2 \pi (k+ sN)}\right )
\right|^2
\\
&&
+
\sum_{M\le |k|\le \frac{N}{2}}
\left|
g^m \sum_{s \ne 0} \hat u^{\infty, 0}\left ( {2 \pi (k+ sN)}\right )
\right|^2,
\end{eqnarray*}
\[
\sum_{|k|\le M}
\left|
\hat u_k (0) e^{\hat q_k t_m} - g^m \hat u_k^{\infty, 0}
\right|^2
 \le t^2 (C_1(h) + C_2 \dt)^2 \|u_0 \|^2,
\]
and
\begin{eqnarray*}
&&\sum_{M\le |k|\le \frac{N}{2}}
\left|
\hat u_k (0) e^{\hat q_k t_m} - g^m \hat u_k^{\infty,0}
\right|^2
+\sum_{|k|>\frac{N}{2}}
\left|
\hat u_k (0) e^{qt_m}
\right|^2\\
&=& C \sum_{|k|>M} \left |\hat u^{\infty, 0} \right |^2
\le C \left (\frac{1}{M} \right )^{2 \nu_1}
\sum_{|k|>M} |k|^{2 \nu_1} \left |\hat u^{\infty, 0} \right |^2
\le  C_3 (\nu_1) \dx^{2\nu_1} \|u_0\|_{H^{\nu_1} (1)}
\end{eqnarray*}
hold.
Also following \cite{SKB02, Str}, there exists $\alpha \in \R$ such that
\begin{eqnarray*}
&&\sum_{|k|\le M}
\left|
g^m \sum_{s \ne 0} \hat u^{\infty, 0}\left ( {2 \pi (k+ sN)}\right )
\right|^2
+
\sum_{M\le |k|\le \frac{N}{2}}
\left|
g^m \sum_{s \ne 0} \hat u^{\infty, 0}\left ( {2 \pi (k+ sN)}\right )
\right|^2
\\
 &\le& \sum_{|k|\le \frac{N}{2}-1}
\left |
\sum_{s\ne 0} \hat u^{\infty,0} \left ( \frac{2 \pi (k+ sN)}{2L}\right )
\right|^2
\le \dx^{2\alpha} C(u_0),
\end{eqnarray*}
gives
\[
\|u(x_j, t_m) - U_j^m\|^2 \le t^2 (C_1(h) + C_2 \dt)^2 \|u_0 \|^2 +
C_3 (\nu_1) \dx^{2\nu_1} \|u_0\|_{H^{\nu_1} (1)}
+
\dx^{2\alpha} C(u_0).
\]
%
Thus we finish this section with the following theorem.
\begin{thrm}
If the initial value problem (\ref{f:find-linearide01}) with some non-smooth initial function $u_0 (x, 0)=g(x)$  is considered so that the condition (\ref{eqq110:f}) holds  and is approximated by the stable one step finite difference formula (\ref{period03:f}), $J, \hat J_k,$ $-\frac{N}{2}+1 \le k\le \frac{N}{2}$ satisfy
 \textbf{A1}~-~\textbf{A5},
then
\[
\|u(x_j, t_m) - U_j^m\| \le t (C_1(h)  + C_2 \dt) \|u_0 \| +
C_2 (\nu_1) \dx^{\nu_1} \|u_0\|_{H^{\nu_1} (1)}
+
\dx^{\alpha} C_3(u_0).
\]
where  $C_1(h)$, $C_2$ and $C_3$ are constants.
\end{thrm}
%
%
\section{Convergence of semidiscrete approximation}
Applying the discrete Fourier transform to (\ref{period02:f})
\begin{equation}\label{Df:find-linearide01}
 {{\frac{d}{dt}{\tilde U_k}}} = \tilde q_k \tilde U_k
\end{equation}
where $\tilde q_k=\left (\tilde J_k-\tilde J_0\right )$ and its solution
can be found as
\begin{equation}\label{Dperiod02:f}
\tilde U_k (t)= e^{\tilde q_k t}\tilde U_k (0).
\end{equation}
Applying the inverse of the DFT in (\ref{Dperiod02:f})
\begin{equation} \label{Dperiod03:f}
U(x_j,t)=\sum_{k=-\frac{N}{2}+1}^{\frac{N}{2}}e^{\tilde q_k t}\tilde U_k (0)
e^{{ik 2\pi x_j}}.
\end{equation}
Thus comparing (\ref{period07:f}) and (\ref{Dperiod03:f})
\begin{eqnarray}\label{Dperiod04:f}
u(x_j,t)-U(x_j,t) & = & \sum_{k=-\infty}^{\infty} \hat u_k (0) e^{\hat q_k t}
e^{{ik2\pi x_j}}
-\sum_{k=-\frac{N}{2}-1}^{\frac{N}{2}} e^{\tilde q_k t} e^{{ik2\pi x_j}}
\tilde U_k (0)
\nonumber\\
&=&
\sum_{k=-\frac{N}{2}-1}^{\frac{N}{2}} e^{{ik2\pi x_j}}
\left (\hat u_k (0) e^{\hat q_k t} -
e^{\tilde q_k t}
\tilde U_k (0)\right )
%
 +\sum_{|k|> \frac{N}{2}} \hat u_k (0) e^{\hat q_k t} e^{{ik2\pi x_j}}.\nonumber\\
\end{eqnarray}
Applying Parseval's relation to (\ref{Dperiod04:f})
\begin{eqnarray}\label{Dperiod05:f}
\left |u(x_j,t)-U(x_j,t)\right |^2
&\le&
\sum_{k=-\frac{N}{2}-1}^{\frac{N}{2}}
\left |\hat u_k (0) e^{\hat q_k t} -
e^{\tilde q_k t}
\tilde U_k (0)\right |^2
 + \sum_{|k|> \frac{N}{2}} \left |
\hat u_k (0) e^{\hat q_k t} e^{{ik2\pi x_j}}\right |^2
\nonumber\\
&\le&
\sum_{k=-\frac{N}{2}-1}^{\frac{N}{2}}
\left |\hat u_k (0) e^{\hat q_k t} -
e^{\tilde q_k t}
\tilde U_k (0)\right |^2
 +  \sum_{|k|> \frac{N}{2}} \left |
\hat u_k (0) e^{\hat q_k t} \right |^2
\end{eqnarray}
Now applying Poisson's formula to the
first part of the right-hand side of (\ref{Dperiod05:f}),
\begin{eqnarray*}
\sum_{k=-\frac{N}{2}+1}^{\frac{N}{2}}\left |\hat u_k (0) e^{\hat q_k t} -
e^{\tilde q_k t}
\tilde U_k (0)\right |^2
&=& \sum_{k=-\frac{N}{2}+1}^{\frac{N}{2}}\left | e^{\hat q_k t}-e^{\tilde q_k t}\right|^2
|\hat u_0^{\infty}({2\pi k})|^2\\
&&+\sum_{k=-\frac{N}{2}+1}^{\frac{N}{2}}
\sum_{s\ne 0} e^{2\tilde q_k t} \left |\hat u_0^{\infty}
\left(  {2\pi(k+sN)}\right) \right |^2.
\end{eqnarray*}
Now 
\[
|e^{\hat q_k t}-e^{\tilde q_k t}|^2=
|e^{\hat q_k t}\left ( 1- e^{t(\tilde q_k -\hat q_k)}\right )|^2
\le t^2 |\hat q_k-\tilde q_k|^2.
\]
Now
\begin{eqnarray*}
\tilde q_k - \hat q_k &=& \left ( \sum_{m=-\infty}^{\infty} \hat q_{k+mN}\right )- \hat q_k, \quad\mbox{using~(\ref{poisson:f})}\\
&=&\sum_{m \ne 0} \hat q_{k+mN}
= \sum_{m \ne 0}\left ( \hat J_{k+mN} -\hat J_{mN} \right ),\quad\mbox{since $\hat q_k=\hat J_k -\hat J_0$,}\\
&=& \sum_{m \ne 0}\left ( \hat J^{\infty}(2\pi(k+mN)) -\hat J^{\infty}(2\pi mN) \right ),
\quad \mbox{using~(\ref{Pnprela:f})}.
\end{eqnarray*}
So from  \cite{SKB02} it follows that
\[
|\tilde q_k - \hat q_k| \le 2 \hat J^{\infty}\left(\frac{\pi}{h}\right )= C_1(h),
\]
and  $C_1(h)\rightarrow 0$ as $h \rightarrow 0$ when $J^{\infty}$ is smooth enough (if $J^\infty \in L_2(\mathbb{R})$, then $|\hat J^\infty (\xi)|\rightarrow 0$ as $|\xi| \rightarrow \infty$ ~\cite{K.Maleknejad},~\cite[page 30]{LNTre}).
Thus
\[
|e^{\hat q_k t}-e^{\tilde q_k t}|^2 \le t^2 C_1^2(h).
\]
So applying the relation (\ref{finiteinfinite:f}), Parseval's relation and 
Theorem~\ref{soveedef:f}
\begin{eqnarray*}
&&\sum_{|k|\le {\frac{N}{2}}}
\left |\left ( e^{\hat q_k t}-e^{\tilde q_k t}\right )\hat u_0^{\infty}\right |^2
\le t^2 C_1^2(h) \sum_{|k|\le {\frac{N}{2}}}
\left |\hat u_k(0)\right |^2\\
&\le& t^2 C_1^2(h) \left ( |\hat u_0(0)|^2 + \sum_{|k| >0} |k|^{0} \left |\hat u_k(0)\right |^2 \right )\\
&\le& t^2 C_1^2 (h)\|u_0 \|_{*,0}^2\equiv t^2 C_1^2 (h)\|u_0 \|_{H^0(1)}^2=t^2 C_1^2 (h)\|u_0 \|^2.
\end{eqnarray*}
Thus, similar to Section~\ref{defineqqq} there exists a constant $C_2(\sigma)$ such that
\[
\left\|u(x_j,t)-U(x_j,t)\right \|^2
\le t^2 C_1^2 (h)\|u_0 \|^2+C_2^2(\sigma) h ^{2\sigma}
\|u_0 \|_{H^\sigma (1)}^2.
\]
Thus we conclude
\begin{thrm}
If the  semidiscrete approximation (\ref{period02:f}) of the initial
value problem (\ref{f:find-linearide01}) is stable $J, \hat J_k,$ $-\frac{N}{2}+1 \le k\le \frac{N}{2}$ satisfy \textbf{A1}~-~\textbf{A5} and
 $u_{0}\in H^{\sigma}({1})$ with $\sigma>\frac{1}{2},$ then there exist
constants $C_1(h),$ $C_2(\sigma)$ such that
\[
\left\|u(x_j,t)-U(x_j,t)\right \|
\le t C_1(h)\|u_0 \| + C_2(\sigma) h ^{\sigma}
\|u_0 \|_{H^\sigma (1)}.
\]
\end{thrm}
%
\section{Numerical Experiments and discussion}
Here we start by experimenting numerical error  in  the approximation
(\ref{period03:f})  of (\ref{f:find-linearide01}).
We compute numerical error at $t=1$. Figure~\ref{accuracy03:fig} shows
the behaviour of $\|u(\cdot, t_n)-U^n(\cdot)\|_h$ for various choices of $h$
and $\dt$ with smooth and non-smooth $u_0$ for all $x\in [0, 1]$ where
$J^\infty(x)=\sqrt{\frac{10}{\pi}}exp(-10 x^2)$. From Figure~\ref{accuracy03:fig}, we observe that for smooth $u_0$ the rate of convergence of the solutions are faster than for the non-smooth $u_0$.
Here we  also notice that choices of $h$ and $\dt$ have an impact on the rate of convergence.
That agrees with our theoretical estimates (and also  motivates us  to investigate further the theoretical stability and convergence analysis
 of such a one step approximation).
\begin{figure}[t]
\begin{center}
{\bf (a) \hspace{5.5cm} (b)}\\
\includegraphics[width=0.49\textwidth,height=7.5cm]{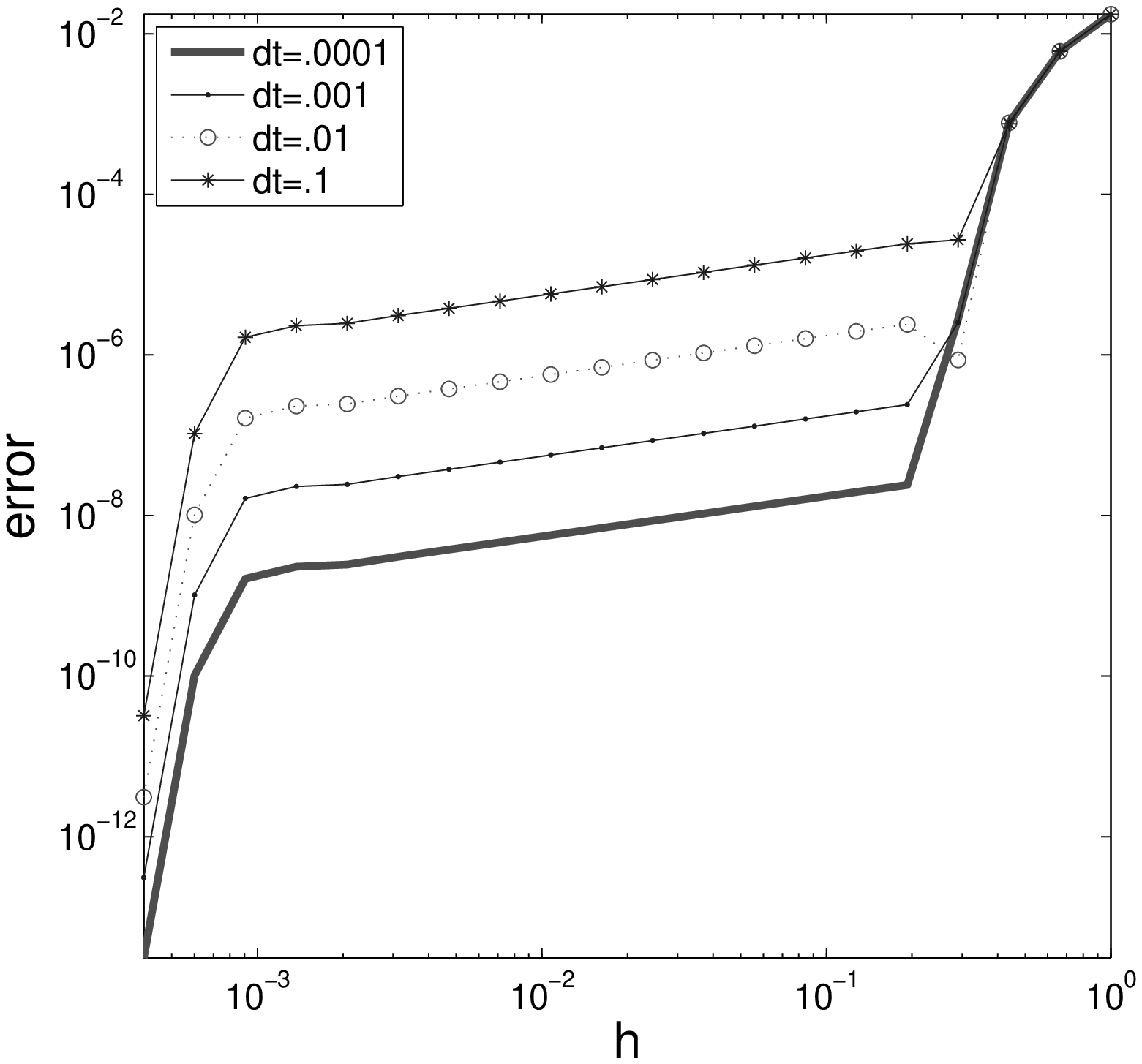}
\includegraphics[width=0.49\textwidth,height=7.5cm]{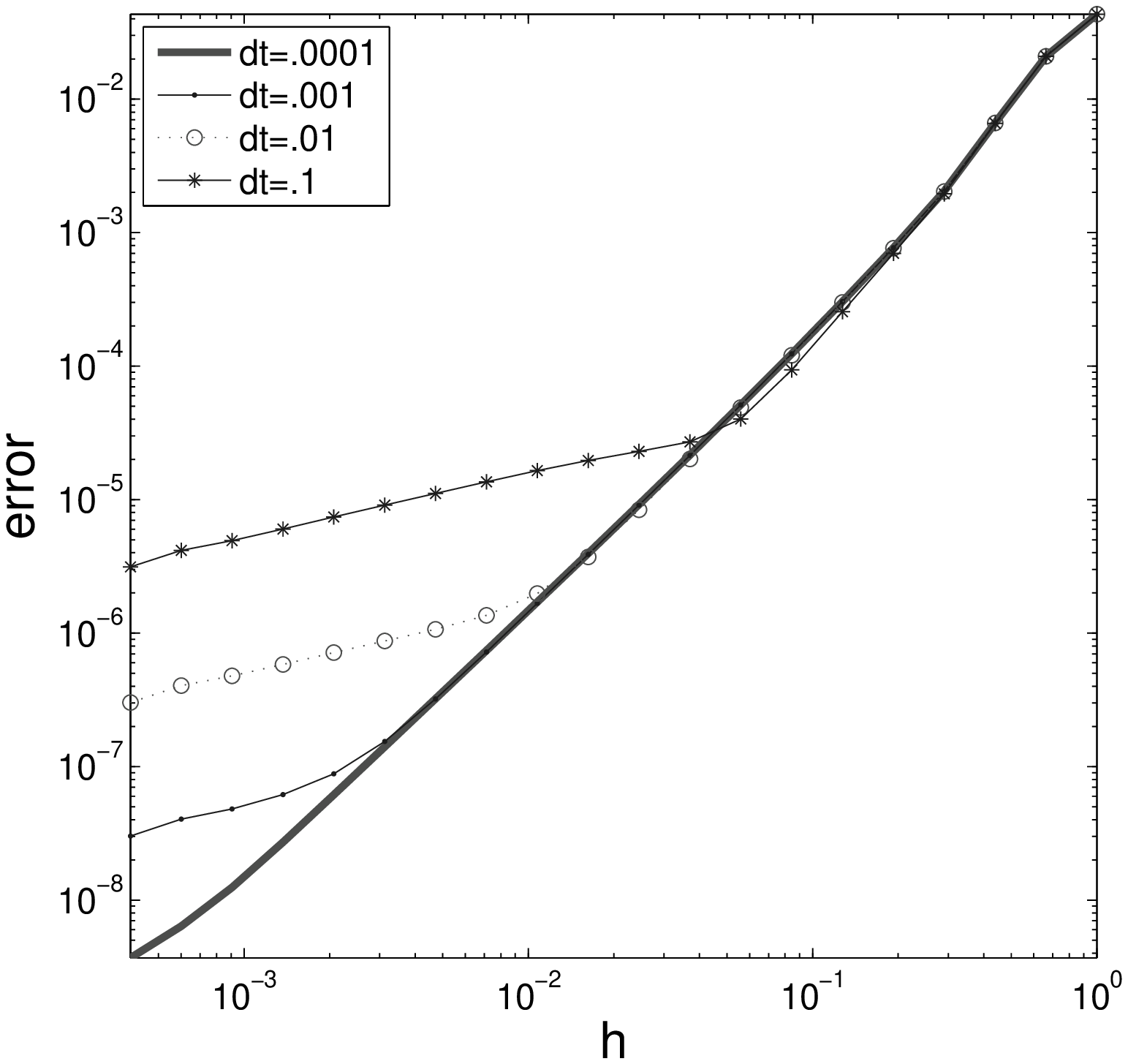}
\caption{Here we plot the error term $\|u(\cdot, t_n)-U^n(\cdot)\|_h$ at $t=1$ by varying $\dt$ and $h$ with
 $J^\infty(x)=\sqrt{\frac{10}{\pi}}e^{-10 x^2}$ and
$(a)$ $u_0 (x)=\sqrt{\frac{1}{\pi}}e^{-(x-\frac{1}{2})^2}$ (left figure),
     $(b)$ $u_0 (x)=\frac{1}{2}e^{-|x-\frac{1}{2}|}$ (right figure).
}
\label{accuracy03:fig}
\end{center}
\end{figure}
From this study we notice that the full discrete scheme is conditionally stable. The accuracy of the scheme depends on the smoothness of the initial function. We have some limitations in this study. We impose some reasonable restrictions on the kernel function to prove the stability and convergence results. The analysis of higher order schemes in one and multi-dimensions leaves as a future study which is of course more challenging.

%

\begin{thebibliography}{10}

\bibitem{Atkinson}
Kendall Atkinson and Weimin Han.
\newblock {\em {Theoretical Numerical Analysis}}.
\newblock Springer, 2001.

\bibitem{A.Chmaj}
Peter~W. Bates and Adam Chmaj.
\newblock A discrete convolution model for phase transitions.
\newblock {\em Archive for Rational Mechanics and Analysis}, 150(4):281--305,
  1999.

\bibitem{C.Fife}
Peter~W. Bates, Paul~C. Fife, Xiaofeng Ren, and Xuefeng Wang.
\newblock Travelling waves in a convolution model for phase transitions.
\newblock {\em Archive for Rational Mechanics and Analysis}, 138(2):105--136,
  July 1997.

\bibitem{SKB02}
Samir~K. Bhowmik.
\newblock {\em Numerical approximation of a nonlinear partial
  integro-differential equation}.
\newblock PhD thesis, Heriot-Watt University, Edinburgh, UK, April, 2008.

\bibitem{SamirKumarBhowmik02}
Samir~Kumar Bhowmik.
\newblock Stability and convergence analysis of a one step approximation of a
  linear partial integro-differential equation.
\newblock {\em Technical Report, KdV Institute for Mathematics, University of
  Amsterdam}, December 2009.

\bibitem{PCFphase02}
C-K Chen and Paul~C. Fife.
\newblock Nonlocal models of phase transitions in solids.
\newblock {\em Advances in Matheamtical Sciences and Applications},
  10:821--849, 2000.

\bibitem{F.Chen}
Fengxin Chen.
\newblock Uniform stability of multidimensional travelling waves for the
  nonlocal {A}llen-{C}ahn equation.
\newblock {\em Fifth Mississippi State Conference on Differential Equations and
  Computational Simulations, Electronic Journal of Differential Equations},
  Conference 10:109--113, 2003.

\bibitem{DJD001}
Daniel~J. Duffy.
\newblock {\em {Finite Difference Methods for Financial Engineering: a Partial
  Differential Equation Approach}}.
\newblock Wiley Finance, 2006.

\bibitem{Dug1}
Dugald~B. Duncan, M.~Grinfeld, and I.~Stoleriu.
\newblock {Approximating a convolution model of phase separation}.
\newblock Conference Talk, Dundee, at May 21, 2003.

\bibitem{Dug}
Dugald~B. Duncan, M~Grinfeld, and I~Stoleriu.
\newblock {Coarsening in an integro-differential model of phase transitions}.
\newblock {\em Euro. Journal of Applied Mathematics}, 11:511--523, 2000.

\bibitem{phase03}
Paul~C. Fife.
\newblock Models of phase separation and their {M}athenmatics.
\newblock {\em Electronic Journal of Differential Equation}, 48:1--26, 2000.

\bibitem{CFphase}
Paul~C. Fife.
\newblock Well-posedness issues for medels of phase transitions with weak
  interaction.
\newblock {\em Nonlinearity}, 14:221--238, 2001.

\bibitem{Troy}
Carlo~R. Laing and William~C. Troy.
\newblock {PDE} methods for nonlocal models.
\newblock {\em SIAM J. Applied dynamical systems}, 2(3):487--516, 2003.

\bibitem{K.Maleknejad}
K.~Maleknejad.
\newblock A comparison of fourier extrapolation methods for numerical solution
  of deconvolution.
\newblock {\em New York Journal of Mathematics}, 183:533--538, 2006.

\bibitem{Str}
J.~C. Strikwerda.
\newblock {\em {Finite Difference Schemes and Partial Differential Equations}}.
\newblock Wadsworth and Brooks, Cole Advanced Books and Software, Pacific
  Grove, California, 1989.

\bibitem{LNTre}
L.~N. Trefethen.
\newblock {\em {Spectral Methods in \mat}}.
\newblock SIAM, Philadelphia, 2000.

\end{thebibliography}

\end {document}